\documentclass{sapm}

\usepackage{graphicx}
\usepackage[colorlinks=true]{hyperref}

\newcommand{\id}{\, \mathrm{d}}
\newcommand{\ee}{\mathrm{e}}
\newcommand{\im}{\mathrm{i}}
\newcommand{\bO}{\mathcal{O}}
\def\ifrac#1#2{\textstyle{\frac{#1}{#2}}\displaystyle}
\def\Ai{{\rm Ai\,}}
\def \arccosh{{\rm arccosh}}
\def\hyp#1#2#3#4{\displaystyle{{}_2F_1\left({#1, #2\atop #3};#4\right)}}
\def\Ohyp#1#2#3#4{\displaystyle{{\bf F}\left({#1, #2\atop #3};#4\right)}}

\jname{STUDIES IN APPLIED MATHEMATICS}
\jvol{AA}
\jyear{YYYY}
\doi{10.1111/((please add article doi))}

\begin{document}

\title[Computation of the coefficients in uniform asymptotics]{Computation of the coefficients appearing in the uniform asymptotic expansions of integrals}

\author[S. Farid Khwaja and A. B. Olde Daalhuis]{Sarah Farid Khwaja and Adri B. Olde Daalhuis\thanks{Address for
correspondence: Dr A.\ B.\ Olde Daalhuis, School of Mathematics,
University of Edinburgh, Edinburgh, EH9 3FD; email:
A.OldeDaalhuis@ed.ac.uk}} \affil{University of Strathclyde, International Study Center, 50 Richmond Street, Glasgow, G1 1XP, UK}
\affil{Maxwell Institute and School of Mathematics, King's Buildings, The University of Edinburgh, Edinburgh  EH9 3FD,  UK}

\maketitle

\begin{abstract}
The coefficients that appear in uniform asymptotic expansions for integrals are typically very complicated. 
In the existing literature the majority of the work only give the first two coefficients. In a limited number of papers where more coefficients are given
the evaluation of the coefficients near the coalescence points is normally highly numerically unstable. In this paper, we illustrate
how well-known Cauchy type integral representations can be used to compute the coefficients in a very
stable and efficient manner. We discuss the cases: (i) two coalescing saddles, (ii) two saddles coalesce with two branch points,
(iii) a saddle point near an endpoint of the interval of integration. As a special case of (ii) we give a 
new uniform asymptotic
expansion for Jacobi polynomials $P_n^{(\alpha,\beta)}(z)$ in terms of Laguerre polynomials $L_n^{(\alpha)}(x)$ as
$n\to\infty$ that holds uniformly for $z$ near $1$. Several numerical illustrations are included.
\end{abstract}

\section{Introduction}
In this paper, we discuss approximations of integrals of the form
\begin{equation}\label{MainInt}
F(\lambda,z)=\int_{\mathcal{C}}\ee^{\lambda f(t,z)}g(t,z)\id t,
\end{equation}
where $\mathcal{C}$ is a contour in the complex plane and $\lambda$ is a large parameter.
The critical points for these type of integrals are saddle points of $f(t,z)$, branch points of the integrand, and possibly end points of the
contour of integration. These critical points will depend on the additional parameter $z$, and we assume that the $N$ relevant critical points
will coalesce when $z=z_*$. Via the so-called Bleistein method (see \cite{Blei66}), one can obtain large $\lambda$ asymptotic expansions that 
hold for $z$ in some neighbourhood of $z_*$. Typically, the coefficients in these uniform asymptotic expansions are very complicated,
and most publications only give the first two coefficients. For more than two decades, we already know that it is possible to obtain relatively
simple Cauchy-type integral representations for these coefficients, see for example \cite{OT94}. In a recent paper \cite{TW14},
the remarkable exponentially convergent properties of the trapezoidal rule for integrals is discussed. In this paper, we discuss how this numerical method can
be used to compute coefficients in the uniform asymptotic expansions. With these results we make the uniform asymptotic expansions useful
for the numerical evaluation of the integrals.

The main steps of the Bleistein method are given in section \ref{Sect1}. We note that the details in that section 
are not correct in all cases in which the Bleistein method can be used, but many cases are covered.
The first step is to bring the integral in canonical form, and afterwards a special integration by parts
will give us the uniform asymptotic expansion. With the method introduced in  \cite{OT94} we obtain the Cauchy-type integral representations for these coefficients,
and then we mention how the ideas of \cite{TW14} can be used to compute the coefficients numerically is a stable and efficient manner.

The most well-known case of uniform asymptotics for integrals is the coalescence of two saddle points. We give the details in section \ref{Sect2}
and include as a numerical example the coefficients of the well-known uniform asymptotic expansion of the Bessel function
$J_{\lambda}\left(\lambda z\right)$, as $\lambda\to\infty$ and $z$ near $1$. In that numerical illustration, we do not only observe that our method works,
but also that exact representation of the coefficients is highly numerically unstable near the coalescence.

In section \ref{Sect3} we give full details for the case of the coalescence of two saddle points with two branch points. The main example
is the Gauss hypergeometric function ${}_2F_1\left(a+\lambda,b-\lambda;c;(1-z)/2\right)$ in which again $\lambda\to\infty$.
For the case $z$ near $1$, we derive a uniform asymptotic expansion in terms of Kummer $M$-functions. For the coefficients we obtain
integral representations, but the results differ slightly from the ones in section \ref{Sect1}. Hence, we give details on how these integral
representations can be derived. Taking special values of the parameters $a$, $b$ and $c$ we obtain a new uniform asymptotic
expansion for Jacobi polynomials $P_n^{(\alpha,\beta)}(z)$ in terms of Laguerre polynomials $L_n^{(\alpha)}(x)$ as
$n\to\infty$ that holds uniformly for $z$ near $1$.

The second case in section \ref{Sect3} is for $z$ near $-1$. Since the details are very similar to the previous case, we
only give a few details of the uniform asymptotic expansion of the Gauss hypergeometric function in terms of Kummer $U$-functions.
Numerical illustrations are provided for both cases. 

The coalescence of a saddle point with the end point of the contour of integration is also important in many applications.
We discuss this case in the final section of the paper, but we give only the main details.


\section{The Bleistein method}\label{Sect1}

Before we can use the Bleistein method we have to bring integral (\ref{MainInt}) in canonical form via a transformation
\begin{equation}\label{Transf}
f(t,z)=p(\tau,\zeta)+p_{0}.
\end{equation}
The new integrand should have a similar critical point structure, and $\zeta$ and $p_0$ are determined by the condition
that the relevant critical points in the $t$-plane are mapped to the ones in the $\tau$-plane.
Often the function $p(\tau,\zeta)$ is a polynomial in $\tau$, but a non-polynomial example will be included in this publication.

The new integral representation is then of the form
\begin{equation}\label{Canon}
 F(\lambda,z)=\frac{\ee^{\lambda p_0}}{2\pi\im}\int_{\widetilde{\mathcal{C}}}\ee^{\lambda p(\tau,\zeta)}q(\tau,\zeta)G_{0}(\tau)\id\tau,
\end{equation}
in which
\begin{equation}\label{G0}
G_0(\tau)=2\pi\im\frac{g(t,z)}{q(\tau,\zeta)}\frac{\id t}{\id\tau},
\end{equation} 
such that
\begin{equation}\label{Approximant}
{\cal A}_n(\lambda,\zeta)=\frac{1}{2\pi\im}\int_{\widetilde{\mathcal{C}}}\ee^{\lambda p(\tau,\zeta)}\tau^n q(\tau,\zeta)\id\tau,\qquad n=0,\dots,N-1,
\end{equation}
are the main approximants. In the case of two coalescing saddle points ${\cal A}_0(\lambda,\zeta)$ is an Airy function, $p(\tau,\zeta)$ is a cubic polynomial
and $q(\tau,\zeta)=1$. See section \ref{Sect2}. In other cases the function $q(\tau,\zeta)$ contains the branch-point behaviour of the non-exponential part of the
integrand. See section \ref{Sect3}, especially (\ref{qG3}).

We note that the following details are not correct in all cases in which the Bleistein method can be used, but they cover many cases. We define for $s=0,1,2,\dots$,
\begin{eqnarray}
G_{s}(\tau)&=&\sum_{n=0}^{N-1}a_{sn}\tau^{n}+p'(\tau,\zeta)H_{s}(\tau),\nonumber \\
G_{s+1}(\tau)&=&\frac{-1}{q(\tau,\zeta)}\frac{\id}{\id\tau}\left(q(\tau)H_{s}(\tau)\right),\label{Gs}
\end{eqnarray}
where $'$ indicates differentiation with respect to $\tau$. Then via integration by parts we have
\begin{equation}\label{UniExp}
 F(\lambda,z)=\ee^{\lambda p_0}\sum_{n=0}^{N-1}{\cal A}_n(\lambda,\zeta)\sum_{s=0}^{S-1}\frac{a_{s,n}}{\lambda^s}+R_S(\lambda,z),
\end{equation}
where
\begin{equation}\label{Remainder}
 R_S(\lambda,z)=\lambda^{-S}\frac{\ee^{\lambda p_0}}{2\pi\im}\int_{\widetilde{\mathcal{C}}}\ee^{\lambda p(\tau,\zeta)}q(\tau,\zeta)G_{s}(\tau)\id\tau.
\end{equation}
Hence, the expansion in (\ref{UniExp}) seems to have an asymptotic property.

The coefficients $a_{s,n}$ can be expressed as,
\begin{equation}\label{coeff.asn}
a_{s,n}=\frac{1}{2\pi\im }\oint_{|\tau|=r} A_{s,n}(\tau)G_{0}(\tau)\id \tau,
\end{equation}
where $A_{s,n}(\tau)$ are simple rational functions that satisfy the following equation
\begin{equation}\label{S1An}
A_{s+1,n}(\tau)=\frac{q(\tau)}{p'(\tau,\zeta)}\frac{\id}{\id\tau}\left(\frac{A_{s,n}(\tau)}{q(\tau)}\right).
\end{equation}
The initial rational functions $A_{0,n}(\tau)$ will be determined on a case by case basis.
The radius $r$ in (\ref{coeff.asn}) has to be chosen in such a way that
all the relevant critical points are encircled by the contour once in the positive sense.

 According to \cite{TW14} the right-hand side of
\begin{equation}\label{treff_eqn}
\frac{1}{2\pi\im }\oint_{|\tau|=r} F(\tau)\id \tau\approx\frac{1}{2M}\sum_{m=0}^{2M-1}w_{m}F(w_{m}),~~~{\rm where}~~ w_{m}=r\ee^{\pi\im  m/M},
\end{equation}
converges exponentially fast to the left-hand side as $M\to\infty$, as long as $F(\tau)$ is analytic in a disc $|\tau|\leq \widetilde{r}$, with $r<\widetilde{r}$.

Applying this approximation to integral representation (\ref{coeff.asn}), we obtain the approximation
\begin{equation}\label{a_s_n}
a_{s,n}\approx \frac{1}{2M}\sum_{m=0}^{2M-1}w_{m}A_{s,n}(w_{m})G_{0}(w_{m}).
\end{equation}
Since the $A_{s,n}(\tau)$ are simple rational functions, the approximation of, say, $a_{10,3}$ is not much harder than that of the first coefficient $a_{0,0}$.
The main data that we need is $G_{0}(w_{m})$, $m=0,\dots,2M-1$. The function $G_0(\tau)$ is defined in (\ref{G0}) and it involves the mapping (\ref{Transf}).
This nonlinear mapping usually causes multivaluedness issues in the complex plane
when we have to determine the $t=t_m$ that corresponds to $\tau=w_m$.
However, starting at $\tau=w_0=r$ it is relatively easy to control the multivaluedness
when we move from point $\tau=w_{m}$ to $\tau=w_{m+1}$. One could use $t_m$ as an initial guess when one tries to determine $t_{m+1}$.

In implementations of these results to approximate the integral via uniform asymptotic expansion (\ref{UniExp}), it also makes sense
to interchange the order of summation via
\begin{equation}\label{fulsum}
\sum_{s=0}^{S-1}\frac{a_{s,n}}{\lambda^s}\approx \frac{1}{2M}\sum_{m=0}^{2M-1}w_{m}\widetilde{A}(w_{m})G_{0}(w_{m}),\quad {\rm where}\quad
\widetilde{A}(\tau)=\sum_{s=0}^{S-1} \lambda^{-s}A_{s,n}(\tau).
\end{equation}


\section{The case of two coalescing saddle points}\label{Sect2}
Now we suppose that the integral (\ref{MainInt}) has two saddle points  located at $t=t_{\pm}$ which depend on a parameter $z$. 
In order to express this integral in its canonical form, we consider the following cubic transformation
\begin{equation}\label{cubic}
f(t,z)=p(\tau,\zeta)+p_0=\ifrac{1}{3}\tau^{3}- \zeta \tau+ p_0,
\end{equation}
suggested by Chester, Friedman and Ursell in 1957 in \cite{CFU53}.
The saddle points $t=t_{\pm}$ should correspond to the saddle points of the cubic polynomial at $\tau=\pm \sqrt{ \zeta}$ in the complex $\tau$-plane.  Thus we have
\begin{equation}\label{parameters}
\ifrac{4}{3} \zeta^{3/2}=f(t_{-},z)-f(t_{+},z),\qquad 2p_0=f(t_{-},z)+f(t_{+},z).
\end{equation}
Substituting the cubic transformation (\ref{cubic}) in the integral (\ref{MainInt}), we obtain
\begin{equation}\label{integral5}
F(\lambda,\zeta)=\frac{\ee^{\lambda p_0}}{2\pi\im}\int_{\mathcal{\widetilde{C}}}\ee^{\lambda \left(\frac{1}{3}\tau^{3}- \zeta \tau\right)}G_{0}(\tau)\id \tau,
\end{equation}
where $\mathcal{\widetilde{C}}$ is the image of the contour $\mathcal{C}$ and
\begin{equation}\label{S2G0}
G_{0}(\tau)=2\pi\im g(t)\frac{\id t}{\id\tau}.
\end{equation}
In this case with $q(\tau)=1$, equations (\ref{Gs}) take the following form
\begin{equation}\label{bm}
 G_{s}(\tau)=a_{s,0}+a_{s,1}\tau+\left(\tau^2- \zeta\right)H_{s}(\tau),\qquad
 G_{s+1}(\tau)=-H'_{s}(\tau).
 \end{equation}
It follows that
\begin{equation}\label{S2coeff}
a_{s,0}=\frac{G_{s}(\sqrt{ \zeta})+G_{s}(-\sqrt{ \zeta})}{2},\qquad a_{s,1}=\frac{G_{s}(\sqrt{ \zeta})-G_{s}(-\sqrt{ \zeta})}{2\sqrt{ \zeta}}.
\end{equation}
However, these representations are not very useful.

For simplicity let us assume that we can deform  the contour $\mathcal{\widetilde{C}}$ to a steepest descent path from $\infty e^{-\pi\im /3}$ to $\infty e^{\pi\im /3}$. If that is the case then the 
we obtain
\begin{eqnarray}\label{approxAii}
F(\lambda, \zeta)&=&\ee^{p_0\lambda}\left(\Ai\left(\lambda^{2/3} \zeta\right)\sum_{s=0}^{S-1}\frac{a_{s,0}}{\lambda^{s+1/3}}
-\Ai'\left(\lambda^{2/3} \zeta\right)\sum_{s=0}^{S-1}\frac{a_{s,1}}{\lambda^{s+2/3}}\right)\nonumber \\
&&+R_S(\lambda , \zeta),
\end{eqnarray}
where $\Ai(x)$ is the {Airy function} and $\Ai'(x)$ is its derivate (see \cite[\href{http://dlmf.nist.gov/9.5.ii}{\S 9.5(ii)}]{NIST:DLMF}), and where
\begin{equation}\label{UAI5}
R_S(\lambda , \zeta)=\lambda^{-S}\frac{\ee^{p_0\lambda}}{2\pi\im}\int_{\mathcal{\widetilde{C}}}\ee^{\lambda \left(\frac{1}{3}\tau^{3}- \zeta \tau\right)}G_{S}(\tau)\id \tau.
\end{equation}

To obtain Cauchy-type integral representations for the coefficients $a_{s,n}$ we use \cite{OT94} and define
\begin{equation}\label{Asn}
A_{0,0}(\tau)=\frac{\tau}{\tau^{2}-\zeta},\quad A_{0,1}(\tau)=\frac{1}{\tau^{2}-\zeta},\quad
A_{s+1,n}(\tau)=\frac{1}{\tau^{2}-\zeta}\frac{\id}{\id\tau}A_{s,n}(\tau)
\end{equation}
for $s=0,1,2,\dots$, and $n=0,1$. We note that using these rational functions that we have for the coefficients integral representation(\ref{coeff.asn}), the contour of integration does not have to be a circle.
It can be any contour that encircles the two saddle points $\tau=\pm\sqrt\zeta$ once in the positive sense, and
such that all singularities of $G_0(\tau)$ are in the exterior of the contour.

Integral representation (\ref{coeff.asn}) can be used to compute the higher  coefficients in the uniform asymptotic expansion. For example, for this case, we have
\begin{equation}
A_{1,0}(\tau)=-\frac{\tau^2+\zeta}{\left(\tau^{2}-\zeta\right)^3},\qquad A_{1,1}(\tau)=\frac{-2\tau}{\left(\tau^{2}-\zeta\right)^3},
\end{equation}
and thus, using (\ref{coeff.asn}), we obtain
\begin{eqnarray}\label{secondcoeff}
a_{1,0}&=&\frac{G_{0}''(-\sqrt{\zeta})-G_{0}''(\sqrt{\zeta})}{8\zeta^{1/2}}+\frac{G_{0}'(\sqrt{\zeta})+G_{0}'(-\sqrt{\zeta})}{8\zeta}-\frac{a_{0,1}}{4\zeta},\nonumber\\
a_{1,1}&=&-\frac{G_{0}''(\sqrt{\zeta})+G_{0}''(-\sqrt{\zeta})}{8\zeta}+\frac{G_{0}'(\sqrt{\zeta})-G_{0}'(-\sqrt{\zeta})}{8\zeta^{3/2}}.
\end{eqnarray}
These representations are numerically unstable when we try to compute these coefficients near $\zeta=0$. Note that when we compute these coefficients via the trapezoidal
rule (\ref{a_s_n}) we encounter no special problems when $\zeta$ is small, since the integration variable is bounded away from the origin.

In \cite{TemmeV02} more of the rational functions $A_{s,n}(\tau)$ are computed and the results are combined with computer algebra and two point Taylor series expansions
to evaluate the coefficients $a_{s,n}$. Two point Taylor series expansions of $G_0(\tau)$ are also discussed in \cite{TemmeL02}. The Cauchy integral representations for the coefficients in these expansions are slightly simpler than (\ref{coeff.asn}), but substituting these expansions in (\ref{integral5})
results in expansions that are more complicated than (\ref{approxAii}). However, also these Cauchy type integral representations can be combined with (\ref{treff_eqn})
to numerically compute the coefficients in a stable manner.

The recent paper \cite{Dunster2017} deals with uniform asymptotic approximations in turning point problems, and they also consider the trapezoidal rule for Cauchy integrals.
If we would write (\ref{approxAii}) as 
\begin{equation}\label{Dunster}
F(\lambda, \zeta)=\ee^{p_0\lambda}\left(\Ai\left(\lambda^{2/3} \zeta\right)A(\lambda, \zeta)
-\Ai'\left(\lambda^{2/3} \zeta\right)B(\lambda, \zeta)\right),
\end{equation}
then the $A(\lambda, \zeta)$ and $B(\lambda, \zeta)$ are so-called coefficient functions.
It is these coefficient functions that are numerical computed in \cite{Dunster2017} via the trapezoidal rule.

\subsection{Numerical illustration}
As a more concrete example we use for the Bessel function the well-known integral representation
(see \cite[\href{http://dlmf.nist.gov/10.9.E17}{10.9.17}]{NIST:DLMF})
\begin{equation}
J_{\lambda}(\lambda z)=\frac{1}{2\pi\im }\int_{\infty-\pi\im }^{\infty+\pi\im }\ee^{\lambda(z\sinh t-t)}\id t.
\end{equation}
Using the notation in (\ref{MainInt}), we have $g(t,z)=1/(2\pi\im)$ and the function $f(t,z)=z\sinh t-t$ has
saddle points at $t=\pm \arccosh(z^{-1})$ which coalesce when $z=1$. Since $f(t,z)$ is odd in $t$ it follows from (\ref{parameters}) that
$p_0=0$ and we have
\begin{eqnarray}\label{Bparameters}
\ifrac{2}{3}\zeta^{2/3}=&\arccosh(1/z)-\sqrt{1-z^{2}},\qquad &0<z\leq1,\nonumber\\
\ifrac{2}{3}\left(-\zeta\right)^{2/3}=&\sqrt{z^{2}-1}-\arccos(1/z),\qquad &z\geq 1.
\end{eqnarray}
Furthermore, it also follows from (\ref{coeff.asn}) that $a_{2s+1,0}=a_{2s,1}=0$.

Uniform asymptotic expansion (\ref{approxAii}) is well-known, see for example
\cite[\href{http://dlmf.nist.gov/10.20.E4}{10.20.4}]{NIST:DLMF}. In this special case, the coefficients are easy to compute
via the methods explained in \cite[\href{http://dlmf.nist.gov/10.20}{\S10.20}]{NIST:DLMF}, and we can compare our numerical
results with the exact results. In the numerical illustration below, we take $z$ close to the  coalescing point, and we did observe
that the `exact results' are highly numerically unstable. 
This numerical instability of the exact coefficients was also observed in \cite{Temme97},
and two methods were introduced to compute the asymptotic expansion (\ref{approxAii}).
Both of these methods involve expansions in powers of $\zeta$, and are useful when $\zeta$ is small.

We take $z=0.995$, close to the coalescing point at $1$, then $\zeta=0.00630908356$. In our approximation (\ref{a_s_n}) we take
$r=1$ and $M=30$. The results are displayed in Table \ref{table:table1}. 
Note that even with such a relatively small $M$ we already obtain
$26$ digits precision in the first two coefficients. Even for $a_{10,0}$ we still have $10$ digits precision, 
and this reduces to $4$ digits for $a_{11,1}$.
However, increasing $M$ from $30$ to $40$ we obtain $23$ digits precision for $a_{11,1}$. 
This illustrates the observation in \cite{TW14}  that the trapezoidal rule for integrals converges exponentially fast.

\begin{table}[ht]
\caption{Numerical results} 
\centering 
\scalebox{1}{
\begin{tabular}{c c c c c} 
\hline\hline 
coefficient& approximation & relative error & $M$ \\ [0.3ex] 
\hline 
$a_{0,0}$ &$\phantom{-}1.2611836781$ &  $6.04\times10^{-27}$ & $30$\\ 
$a_{1,1}$ &$-0.0227706934$&  $1.31\times10^{-26}$ & $30$\\
$a_{8,0}$ & $\phantom{-}0.0004795112$ & $2.90\times10^{-19}$ & $30$\\
$a_{9,1}$ & $-0.0009398771$ & $1.43\times10^{-13}$ & $30$\\
$a_{10,0}$ & $-0.0008770417$ & $7.23\times10^{-10}$ & $30$\\
$a_{11,1}$ & $\phantom{-}0.0023029221$ & $1.56\times10^{-4}$ & $30$\\ 
$a_{11,1}$ & $\phantom{-}0.0023032804$ & $6.11\times10^{-23}$ & $40$\\[1ex] 
\hline 
\end{tabular}
}
\label{table:table1} 
\end{table}


\section{The case of coalescence of two saddle points with two branch points}\label{Sect3}
In this section, we consider the Gauss hypergeometric function
\begin{equation}\label{hyp}
\hyp{a+\lambda}{b-\lambda}{c}{\frac{1-z}{2}},
\end{equation}
where $\lambda \to \infty$. The results are related to the limits
\begin{equation}\label{S3limits}
\lim_{\lambda,\mu\to\infty}\hyp{\lambda}{\mu}{\nu+1}{\frac{x}{\lambda\mu}}=\lim_{\mu\to\infty}M\left(\mu,\nu+1,\frac{x}{\mu}\right)
=\frac{\Gamma(\nu+1)}{x^{\nu/2}} I_\nu\left(2\sqrt{x}\right).
\end{equation}
From these limits it follows that we expect interesting behaviour when the variable is small, that is in the case of (\ref{hyp}), 
when $z$ is close to $1$.
In \cite{KwOD14} the authors derive a large $\lambda$ asymptotic expansion in terms of modified Bessel functions. 
In that paper the derivation was based on
integral representation (\ref{S3Int1}) in which two saddle points coalesce with two branch points as $z\to 1$. 
One of the branch points was ignored since the integrand
was exponentially small near that point, and by considering just the coalescence  of two saddle points 
with one branch point, a uniform asymptotic expansion in terms
of modified Bessel functions could be obtained. That expansion was already known from the theory of differential equations, 
see \cite{Jones01}.

Here we deal with the same integral representation, but now we really consider the coalescence of two saddle points 
with two branch points. The uniform
asymptotic expansion that we derive in subsection \ref{S3a}  is in terms of the Kummer $M$-function, 
and holds uniformly for $z$ close to 1. 
Hence, the result is connected to the first equality sign in (\ref{S3limits}). 
The asymptotic expansion will break down near the singularity of (\ref{hyp}),
that is, at $z=-1$. For $z$ close to $-1$, we give in subsection \ref{S3d} a uniform asymptotic expansion 
in terms of the Kummer $U$-function.
Since the derivation is very similar to the previous subsections we only give the main details.

\subsection{$z$ close to $1$}\label{S3a}
For $0<z<1$, we start with the following integral representation (combine
\href{http://dlmf.nist.gov/15.8.E1}{15.8.1} with \href{http://dlmf.nist.gov/15.6.E3}{15.6.3} in \cite{NIST:DLMF})
\begin{equation}\label{S3Int1}
\hyp{a+\lambda}{b-\lambda}{c}{\frac{1-z}{2}}=\frac{L}{2 \pi\im}\int_{\infty}^{(0+)}\frac{\ee^{(\lambda+a-c) \pi\im }\left(\tau+1\right)^{b-c-\lambda}}{\tau^{\lambda+a-c+1}\left(\tau+\frac{1+z}{2}\right)^{b-\lambda} }\id \tau,
\end{equation}
where
 \begin{equation}\label{S3L}
 L=\frac{\Gamma(c)\Gamma(\lambda+1+a-c)}{ \Gamma(\lambda +a)}.
 \end{equation}
Using $\tau=e^{\pi\im  }t$ in (\ref{S3Int1}), we obtain
\begin{equation}\label{S3Int2}
\hyp{a+\lambda}{b-\lambda}{c}{\frac{1-z}{2}}=\frac{L}{2 \pi\im }\int_{-\infty}^{(0+)}\ee^{\lambda f(t)} g(t) \id t,
\end{equation}
where
\begin{equation}\label{S3IntDef}
f(t)=\ln\left(\frac{\frac{1+z}{2}-t}{1-t}\right)-\ln t,\qquad g(t)=\frac{t^{c-a-1}\left(1-t\right)^{b-c}}{\left(\frac{1+z}{2}-t\right)^{b}}.
 \end{equation}
Here the path of integration starts at $\ee^{-\pi\im }\infty$ encircles $0$ once in the positive direction 
and returns to $\ee^{\pi\im }\infty$. 
The points $1$ and $\frac{z+1}{2}$ lie outside the contour of integration. 
For $f(t)$ we choose branch cuts between $t=\frac{1+z}{2}$ and $t=1$ and the negative real axis.
 Using $z=\cos \theta$, the saddle points are located at
 \begin{equation}\label{C2Definespa}
t_{\pm}=\frac{1+\ee^{\pm \im \theta}}{2}.
\end{equation}
The branch points of the phase function are $t=0$, $t=1$ and $t=\frac{z+1}{2}$. Note that these saddle points coalesce when $\theta=0$ with two branch points at $t=1$.

To obtain a uniform asymptotic expansion, we use the transformation,
\begin{equation}\label{S3p}
f(t)=p(\tau,\zeta)+p_0=\ln\left(\frac{\tau-2\zeta}{\tau}\right)+\tau+p_0.
\end{equation}
We take $\zeta=1-\cos\sigma$. For the function $p(\tau,\zeta)$ the saddle points are located at
\begin{equation}
\tau_{\pm}=1-\ee^{\pm \im\sigma},
\end{equation}
and we will insist that these correspond to $t=t_{\pm}$. This gives us
\begin{equation}\label{S3fttau}
f(t_{\pm})=\mp\im\theta=1\mp\im\sigma-\ee^{\pm  \im\sigma}+p_0,
\end{equation}
where we obtain  $p_0=-\zeta$ and $\theta=\sigma+\sin \sigma$.

With the transformation (\ref{S3p}), we obtain integral representation
\begin{equation}\label{S3Int3}
\hyp{a+\lambda}{b-\lambda}{c}{\frac{1-z}{2}}=\frac{L\ee^{-\lambda\zeta}}{2 \pi\im }\int_{-\infty}^{(0+,2\zeta+)}\frac{\ee^{\lambda\tau}\left(\tau-2\zeta\right)^{\lambda-b}}{\tau^{\lambda-b+c}}G_{0}(\tau)\id \tau,
\end{equation}
with
\begin{equation}\label{qG3}
q(\tau)=\frac{\tau^{b-c}}{\left(\tau-2\zeta\right)^b},\qquad
G_{0}(\tau)=\left(\frac{\tau-2\zeta}{\frac{1+z}{2}-t}\right)^{b}\left(\frac{1-t}{\tau}\right)^{b-c}t^{c-a-1}\frac{\id t}{\id\tau},
\end{equation}
where we will need
\begin{equation}\label{dttau}
\frac{\id t}{\id\tau}(\tau_\pm)=\sqrt{\frac{p''(\tau_{\pm},\zeta)}{f''(t_{\pm})}},
\end{equation}
which can be obtained via l'H\^{o}pital's method. Hence, we need
\begin{equation}\label{fpdd}
f''(t_{\pm})=\pm \im\frac{4\ee^{\mp \im\theta}}{ \sin \theta},\qquad p''(\tau_{\pm},\zeta)=\pm \im\frac{\sin \sigma}{1-\cos \sigma}.
\end{equation}

Now (\ref{Gs}) reads
\begin{eqnarray}
G_{s}(\tau)&=&a_{s,0}+a_{s,1}\tau+\frac{\tau^2-2\zeta\tau+2\zeta}{\tau\left(\tau-2\zeta\right)}H_{s}(\tau)\label{S3Gs1}\\
G_{s+1}(\tau)&=&-\frac{\left(\tau-2\zeta\right)^b}{\tau^{b-c}}\frac{\id}{\id\tau}\left(\frac{\tau^{b-c}}{\left(\tau-2\zeta\right)^b}H_{s}(\tau)\right),\label{S3Gs2}
\end{eqnarray}
where
\begin{equation}\label{anbn}
a_{s,1}=\frac{G_{s}(\tau_{+})-G_{s}(\tau_{-})}{\tau_+-\tau_-},\qquad a_{s,0}=G_{s}(\tau_{+})-\tau_{+} a_{s,1}.
\end{equation}
For $s=0$, we have
\begin{equation}
G_{0}(\tau_{\pm})=\ee^{\pm \im\left(\left(\frac{c}{2}-b\right)\sigma+(b-a)\frac{\theta}{2}\right)}R,
\end{equation}
and
\begin{equation}
R=\left(\frac{2\sin(\sigma/2)}{\sin(\theta/2)}\right)^{c-\frac{1}{2}}\sqrt{\cos\left(\frac{\sigma}{2}\right)}\left(\cos\left(\frac{\theta}{2}\right)\right)^{c-a-b-\frac{1}{2}}.
\end{equation}
Thus the first two coefficients are
\begin{eqnarray}\label{First2}
a_{0,0}&=&\frac{\cos\left(\left(b-\frac{c-1}{2}\right)\sigma-(b-a)\frac{\theta}{2}\right)}{\cos \left(\sigma/2\right)} R,\nonumber\\
a_{0,1}&=&\frac{\sin\left(\left(b-\frac{c}{2}\right)\sigma-(b-a)\frac{\theta}{2}\right)}{\sin \sigma} R.
\end{eqnarray}
Using the  integral representation \cite[\href{http://dlmf.nist.gov/13.4.E13}{13.4.13}]{NIST:DLMF} for the Kummer $M$-function we obtain
\begin{eqnarray}\label{C2int2aa}
\hyp{a+\lambda}{b-\lambda}{c}{\frac{1-z}{2}}&\sim& L\ee^{-\lambda\zeta}
\left(\frac{\lambda^{c-1}}{\Gamma(c)}{ M}(b-\lambda,c,2\zeta\lambda)\sum_{s=0}^\infty \frac{a_{s,0}}{\lambda^s}\right.\nonumber\\
&&\left.+\frac{\lambda^{c-2}}{\Gamma(c-1)}{ M}(b-\lambda,c-1,2\zeta\lambda)
\sum_{s=0}^\infty \frac{a_{s,1}}{\lambda^s}\right),\nonumber\\
&&
\end{eqnarray}
as $|\lambda|\to\infty$. This expansion is created such that it holds uniformly for $z$ near $1$. 
It will hold for bounded $z$, where we also have to bound $z$ away from $-1$, 
but it is probably not very practical when $z$ moves too far from $1$ since
the tracking of the multivaluedness is not straightforward.

In terms of the `Olver' functions (see \cite[\href{http://dlmf.nist.gov/13.2.E3}{13.2.3} and \href{http://dlmf.nist.gov/15.2.E2}{15.2.2}]{NIST:DLMF}) this result reads
\begin{eqnarray}\label{C3int3}
&&\Ohyp{a+\lambda}{b-\lambda}{c}{\frac{1-z}{2}}\nonumber\\
&&\qquad\quad\sim \frac{\Gamma(\lambda+1+a-c)}{ \Gamma(\lambda +a)\ee^{\lambda\zeta}}
\left({\bf M}(b-\lambda,c,2\zeta\lambda)\sum_{s=0}^\infty \frac{a_{s,0}}{\lambda^{s+1-c}}\right.\nonumber\\
&&\qquad\qquad\qquad\qquad\qquad\quad\left.+{\bf M}(b-\lambda,c-1,2\zeta\lambda)
\sum_{s=0}^\infty \frac{a_{s,1}}{\lambda^{s+2-c}}\right).
\end{eqnarray}
Taking in (\ref{C2int2aa}) $a={\alpha+\beta+1}$, $b=0$, $c=\alpha+1$ and $\lambda=n$ will give us
a uniform asymptotic approximation for the Jacobi polynomials
\begin{equation}\label{Jacobi}
P_n^{(\alpha,\beta)}(z)=\frac{(\alpha+1)_n}{n!}\hyp{-n}{n+\alpha+\beta+1}{\alpha+1}{{\frac{1-z}2}},
\end{equation}
in terms of Laguerre polynomials
\begin{equation}\label{Laguerre}
L_n^{(\alpha)}(x)=\frac{(\alpha+1)_n}{n!}M\left(-n,\alpha+1,x\right).
\end{equation}
We can present the result as
\begin{eqnarray}\label{JacobiLaguerre}
P_n^{(\alpha,\beta)}(z)&\sim&\frac{\Gamma(n+\beta+1)\ee^{-n\zeta}}{\Gamma(n+\alpha+\beta+1)}
\left(L_n^{(\alpha)}(2\zeta n)\sum_{s=0}^\infty \frac{a_{s,0}}{n^{s-\alpha}}\right.\nonumber\\
&&\qquad\qquad\quad\left.+(n+\alpha)L_n^{(\alpha-1)}(2\zeta n)\sum_{s=0}^\infty \frac{a_{s,1}}{n^{s-\alpha+1}}\right),
\end{eqnarray}
as $n\to\infty$ uniformly for $z$ near $1$.

\subsection{Coefficients}\label{S3b}
In order to find the coefficients $a_{s,n}$  in (\ref{treff_eqn}), we make use of the method introduced in \cite{OT94} and propose the following result.
 
\begin{corollary}

Let
\begin{equation}\label{S3A0B0}
A_{0,0}(\tau,\zeta)=\frac{\tau-2\zeta}{\tau^{2}-2\zeta \tau+2\zeta},\qquad A_{0,1}(\tau,\zeta)=\frac{1}{\tau^{2}-2\zeta \tau+2\zeta},
\end{equation}
and for $s=0,1,2,\dots$, $n=0,1$, let
\begin{equation}\label{S3An}
A_{s+1,n}(\tau,\zeta)=\frac{\left(\tau-2\zeta\right)^{1-b}\tau^{b-c+1}}{\tau^{2}-2\zeta \tau+2\zeta}\frac{\id}{\id\tau}\left(\frac{\left(\tau-2\zeta\right)^{b}}{\tau^{b-c}}A_{s,n}(\tau,\zeta)\right).
\end{equation}
Then the coefficients $a_{s,1}$ have integral representation (\ref{coeff.asn}) with $n=1$, but for $a_{s,0}$ we have
\begin{equation}\label{AB0s1}
a_{s,0}=\frac{1}{2\pi\im }\int_{\mathcal{C}}G_{0}(\tau)\left(A_{s,0}(\tau,\zeta)+(c-1)A_{s-1,1}(\tau,\zeta)\right)\id \tau.
\end{equation}
\end{corollary}

\begin{proof}

Since $G_{0}(\tau)$ is analytic near the saddle points $\tau_{\pm}$, we can write the first expression in (\ref{anbn}) as
\begin{eqnarray}\label{As1}
a_{s,1}&=&\frac{1}{\tau_+ -\tau_-}\frac{1}{2\pi\im }\int_{\mathcal{C}}\left(\frac{G_{s}(\tau)}{\tau-\tau_{+}}-\frac{G_{s}(\tau)}{\tau-\tau_{-}}\right)\id \tau\nonumber \\
&=&\frac{1}{2\pi\im }\int_{\mathcal{C}}G_{s}(\tau)A_{0,1}(\tau,\zeta)\id \tau,
\end{eqnarray}
and combining this result with the second expression in (\ref{anbn}) gives us
\begin{eqnarray}\label{As0}
a_{s,0}&=&\frac{1}{1+\ee^{\im\sigma}}\frac{1}{2\pi\im }\int_{\mathcal{C}}\left(\frac{G_{s}(\tau)}{\tau-\tau_{+}}+\ee^{\im\sigma}\frac{G_{s}(\tau)}{\tau-\tau_{-}}\right)\id \tau\nonumber \\
&=&\frac{1}{2\pi\im }\int_{\mathcal{C}}G_{s}(\tau)A_{0,0}(\tau,\zeta)\id \tau,
\end{eqnarray}
where $\mathcal{C}$ is a simple closed contour which encircles the points $\tau_{\pm}$.
Using (\ref{S3Gs2}) with $s$ replaced by $s-1$ and then integration by parts one obtains
\begin{eqnarray}\label{As-10}
a_{s,0}&=&\frac{1}{2\pi\im }\int_{\mathcal{C}}\frac{\tau^{b-c}}{\left(\tau-2\zeta\right)^{b}}H_{s-1}(\tau)\frac{\id}{\id\tau}\left(\frac{\left(\tau-2\zeta\right)^{b}}{\tau^{b-c}}A_{0,0}(\tau,\zeta)\right)\id \tau\nonumber \\
&=&\frac{1}{2\pi\im }\int_{\mathcal{C}}\frac{\tau^{2}-2\zeta \tau+2\zeta}{\tau\left(\tau-2\zeta\right)}H_{s-1}(\tau)A_{1,0}(\tau,\zeta)\id \tau\nonumber \\
&=&\frac{1}{2\pi\im }\int_{\mathcal{C}}\left(G_{s-1}(\tau)-a_{s-1,0}-a_{s-1,1}\tau\right)A_{1,0}(\tau,\zeta)\id \tau,
\end{eqnarray}
where in the second line we are using (\ref{S3An}) and in the third line we use (\ref{S3Gs1}). Since $A_{1,0}(\tau,\zeta)\sim (c-1)/\tau^2$ as $\tau\to\infty$ we have
\begin{equation}\label{As-11}
a_{s,0}=(c-1)a_{s-1,1}+\frac{1}{2\pi\im }\int_{\mathcal{C}}G_{s-1}(\tau)A_{1,0}(\tau,\zeta)\id \tau.
\end{equation}
We can continue with this process and use the fact that for $s=2,3,4,\dots$, we have $A_{s,0}(\tau,\zeta)=\bO\left(\tau^{-3}\right)$ as $\tau\to\infty$. The result is
\begin{equation}\label{A0s}
a_{s,0}=(c-1)a_{s-1,1}+\frac{1}{2\pi\im }\int_{\mathcal{C}}G_{0}(\tau)A_{s,0}(\tau,\zeta)\id \tau.
\end{equation}
Hence, this result differs from (\ref{coeff.asn}) in the case $n=0$. 

For the case $n=1$ the details are very similar, but since for $s=1,2,3,\dots$, we have $A_{s,1}(\tau,\zeta)=\bO\left(\tau^{-3}\right)$, as $\tau\to\infty$,
we do obtain (\ref{coeff.asn}). Combining that result with (\ref{A0s}) will give us (\ref{AB0s1}).
\end{proof}

\subsection{Numerical illustration}\label{S3c}
We check our approximation for the coefficients by using them in uniform asymptotic approximation (\ref{C2int2aa}). We take $a=b=c=\frac12$, $z=0.9$ and $\lambda=20$.
The corresponding $\zeta=0.025536930$. In the calculation of the coefficients via (\ref{a_s_n}) we take again $M=30$ and $r=1$.
The results are displayed in Table \ref{table:table2}. As mentioned before, in most publication on uniform asymptotics for integrals, the authors only give
$a_{0,0}$ and $a_{0,1}$, like we do in (\ref{First2}). Here we illustrate that it is now possible to take many more terms and obtain much better approximations.
\begin{table}[ht]
\caption{Numerical results} 
\centering 
\scalebox{0.95}{
\begin{tabular}{c c c c c} 
\hline\hline 
$s$& $a_{s,0}$ & $a_{s,1}$ &Asymp approx (\ref{C2int2aa}) & Relative error \\ [0.3ex] 
\hline 
0 &$ \phantom{-}1.014466976$ & $\phantom{-}0.257732955$ & $-0.943360075$& $7.91\times10^{-5}$\\ 
1 &$ \phantom{-}0.000203043$ & $-0.016300238$ & $-0.943287124$& $1.80\times10^{-6}$\\
2 &$ \phantom{-}0.000402104$ & $-0.010444854$ & $-0.943285401$& $2.84\times10^{-8}$ \\
3 &$ \phantom{-}0.000068198$ & $\phantom{-}0.001403704$ & $-0.943285427$& $1.47\times10^{-9}$ \\
4 &$ -0.000133905$ & $\phantom{-}0.003465157$ & $-0.943285428$& $4.34\times10^{-11}$ \\ [1ex] 
\hline 
\end{tabular}
}
\label{table:table2} 
\end{table} 

\subsection{ $z$ close to $-1$}\label{S3d}
For $-1<z<0$, we start with integral representation (\ref{S3Int1}) which we present as
\begin{equation}\label{S4Int1}
\hyp{a+\lambda}{b-\lambda}{c}{\frac{1-z}{2}}=\frac{L}{2\pi\im }\int_{\infty}^{(0+)}\ee^{\lambda f(t)}g(t) \id t ,
\end{equation}
where $\Re(a+\lambda)>0$, $c-a-\lambda\neq 1,2,3,\cdots$, and where
\begin{equation}\label{C2bDefineff}
f(t)=\ln\left(1+\frac{1+z}{2t}\right)-\ln \left(1+t\right),\qquad g(t)=\frac{t^{c-a-1}\left(1+t\right)^{b-c}}{\left(\frac{1+z}{2}+t\right)^{b}},
\end{equation}
and
\begin{equation}\label{S4L}
L=\frac{\Gamma(c)\Gamma(\lambda+1+a-c)}{\ee^{\left(c-a-\lambda\right)\pi\im }\Gamma(\lambda+a)}.\qquad
\end{equation}
We choose  the branch cuts of the phase function $f(t)$ between the points $t=-\frac{1+z}{2}$ and $t=0$ and the half line $t<-1$. Using $z=-\cos \theta$, the saddle points are located at
 \begin{equation}\label{C2Definespb}
t_{\pm}=\frac{\ee^{\pm \im\theta}-1}{2}.
\end{equation}
when $\theta=0$ then $z=-1$ and the two saddle points and two of the branch points will coalesce at $t=0$.

To obtain a uniform asymptotic expansion, we use the transformation,
\begin{equation}\label{S4p}
f(t)=p(\tau,\zeta)+p_0=\ln\left(\frac{\tau+2\zeta}{\tau}\right)-\tau+p_0.
\end{equation}
We take $\zeta=1-\cos\sigma$. For the function $p(\tau,\zeta)$ the saddle points are located at
\begin{equation}
\tau_{\pm}=\ee^{\pm \im\sigma}-1,
\end{equation}
and we will insist that these correspond to $t=t_{\pm}$. Again this will give us (\ref{S3fttau}) and again we have $p_0=-\zeta$ and $\theta=\sigma+\sin \sigma$.

With the transformation (\ref{S4p}), we obtain integral representation
\begin{eqnarray}\label{S3Int3a}
&&\hyp{a+\lambda}{b-\lambda}{c}{\frac{1-z}{2}}=\frac{L\ee^{-\lambda\zeta}}{2\pi\im}\int_{\infty}^{(0+)}\frac{\ee^{-\lambda\tau}\left(\tau+2\zeta\right)^{\lambda-b}}{\tau^{\lambda-c+a+1}}G_{0}(\tau)\id \tau,
\end{eqnarray}
with
\begin{equation}
q(\tau)=\frac{\tau^{c-a-1}}{\left(\tau+2\zeta\right)^b},\qquad
G_{0}(\tau)=\frac{\left(t/\tau\right)^{c-a-1}}{\left({1+t}\right)^{c-b}}\left(\frac{\tau+2\zeta}{\frac{1+z}{2}+t}\right)^{b}\frac{\id t}{\id\tau}.
\end{equation}
The functions $G_s(\tau)$ are defined via (\ref{Gs}) and the coefficients $a_{s,n}$ are defined in (\ref{anbn}).
Again $\frac{\id t}{\id\tau}(\tau_\pm)$ follows from (\ref{dttau}) and (\ref{fpdd}). Thus for $s=0$, we have

\begin{equation}
G_{0}(\tau_{\pm})=\ee^{\pm \im\left((b-a)\frac{\theta}{2}-\left(b+c-a-1\right)\frac{\sigma}{2}\right)}R,
\end{equation}
where
\begin{equation}
R=\left(\frac{2\sin(\sigma/2)}{\sin(\theta/2)}\right)^{b+a-c+\frac{1}{2}}\sqrt{\cos\left(\frac{\sigma}{2}\right)}\left(\cos\left(\frac{\theta}{2}\right)\right)^{\frac{1}{2}-c}.
\end{equation}
For the first two coefficients we have again exact representations
\begin{eqnarray}\label{S4a00}
a_{0,0}&=&\frac{\cos\left(\left(b-a\right)\frac{\theta}{2}-(b-a+c)\frac{\sigma}{2}\right)}{\cos \left(\frac{\sigma}{2}\right)} R,\nonumber\\
a_{0,1}&=&\frac{\sin\left(\left(b-a\right)\frac{\theta}{2}-(b-a+c-1)\frac{\sigma}{2}\right)}{\sin \sigma} R.
\end{eqnarray}
To find the coefficients $a_{s,n}$  we define the rational functions
\begin{equation}\label{S4A0B0}
A_{0,0}(\tau,\zeta)=\frac{\tau+2\zeta}{\tau^{2}+2\zeta \tau+2\zeta},\qquad A_{0,1}(\tau,\zeta)=\frac{1}{\tau^{2}+2\zeta \tau+2\zeta},
\end{equation}
and the other rational functions follow again from (\ref{S1An}).

For the case $n=1$ integral representation (\ref{coeff.asn}) still holds, and for the case $n=0$ we have
\begin{equation}\label{S4AB0s}
a_{s,0}=\frac{1}{2\pi\im }\int_{\mathcal{C}}G_{0}(\tau)\left(A_{s,0}(\tau,\zeta)+(a+b-c)A_{s-1,1}(\tau,\zeta)\right)\id \tau.
\end{equation}
The details for the derivation are similar to the previous case.

Combining  integral representation \cite[\href{http://dlmf.nist.gov/13.4.E14}{13.4.14}]{NIST:DLMF} with tranformation \cite[\href{http://dlmf.nist.gov/13.2.E40}{13.2.40}]{NIST:DLMF} will give us
\begin{eqnarray}\label{S4C2int2a}
&&\hyp{a+\lambda}{b-\lambda}{c}{\frac{1-z}{2}}\nonumber\\
&&\quad\sim \frac{\Gamma(c)\ee^{-\lambda\zeta}}{ \Gamma(a+\lambda)}\left(
U(b-\lambda,a+b-c+1,2\zeta\lambda)\sum_{s=0}^\infty \frac{a_{s,0}}{\lambda^{s+c-a-b}}\right.\nonumber\\
&&\qquad\left.-(\lambda+a-c)U(b-\lambda,a+b-c,2\zeta\lambda)\sum_{s=0}^\infty \frac{a_{s,1}}{\lambda^{s+c-a-b+1}}\right),
\end{eqnarray}
as $|\lambda|\to\infty$.

\subsection{Numerical illustration}\label{S3e}
We check our approximation for the coefficients by using them in uniform asymptotic approximation (\ref{S4C2int2a}). We take $a=b=c=\frac12$, $z=-0.9$ and $\lambda=20$.
The corresponding $\zeta=0.025536930$. In the calculation of the coefficients via (\ref{a_s_n}) we take again $M=30$ and $r=1$.
The results are displayed in table \ref{table:table3}. 
\begin{table}[ht]
\caption{Numerical results} 
\centering 
\scalebox{0.95}{
\begin{tabular}{c c c c c} 
\hline\hline 
$s$& $a_{s,0}$ & $a_{s,1}$ &Asymp approx (\ref{S4C2int2a}) & Relative error \\ [0.3ex] 
\hline 
0 &$ \phantom{-}1.012308968$ & $\phantom{-}0.253890293$ & $-4.111109900$& $1.40\times10^{-4}$\\ 
1 &$ \phantom{-}0.001026614$ & $-0.015679076$ & $-4.111675974$& $2.40\times10^{-6}$\\
2 &$ -0.000137782$ & $-0.010312760$ & $-4.111686107$& $6.11\times10^{-8}$ \\
3 &$ -0.000330251$ & $\phantom{-}0.001268511$ & $-4.111685864$& $1.93\times10^{-9}$ \\
4 &$ \phantom{-}0.000047172$ & $\phantom{-}0.003421650$ & $-4.111685855$& $1.02\times10^{-10}$ \\ [1ex] 
\hline 
\end{tabular}
}
\label{table:table3} 
\end{table}


\section{Saddle point near the end point of the interval}\label{Sect4}
According to \cite[Chapter 22]{Temme15} the canonical form is
\begin{equation}\label{S5canon}
F_\beta(\lambda,\zeta)=\frac1{\Gamma(\beta)}\int_0^\infty \tau^{\beta-1}\ee^{\lambda\left(\zeta\tau-\frac12\tau^2\right)} G_0(\tau)\id\tau,
\end{equation}
and the integration by parts trick is
\begin{eqnarray}\label{S5Gs}
G_{s}(\tau)&=&a_{s,0}+a_{s,1}\tau+(\tau-\zeta)\tau H_{s}(\tau),\nonumber \\
G_{s+1}(\tau)&=&\tau^{1-\beta}\frac{\id}{\id\tau}\left(\tau^{\beta}H_{s}(\tau)\right).
\end{eqnarray}
Note that this differs slightly from (\ref{Gs}). Using (\ref{S5Gs}) in (\ref{S5canon}) produces the expansion
\begin{eqnarray}\label{S5Uniform}
F_\beta(\lambda,\zeta)&\sim& \ee^{\zeta^2\lambda/4}\left(U\left(\beta-\ifrac12,-\zeta\sqrt\lambda\right)\sum_{s=0}^\infty \frac{a_{s,0}}{\lambda^{(2s+\beta)/2}}\right.\nonumber\\
&&\qquad\qquad\left.+\beta U\left(\beta+\ifrac12,-\zeta\sqrt\lambda\right)\sum_{s=0}^\infty \frac{a_{s,1}}{\lambda^{(2s+\beta+1)/2}}\right),
\end{eqnarray}
where for the parabolic cylinder function $U(a,z)$ we have used integral representation \cite[\href{http://dlmf.nist.gov/12.5.E1}{12.5.1}]{NIST:DLMF}.
The coefficients are defined via,
\begin{eqnarray}\label{S5asGs}
a_{s,0}=G_s(0),\qquad a_{s,1}=\frac{G_s(\zeta)-G_s(0)}{\zeta}.
\end{eqnarray}
Again integral representations (\ref{coeff.asn}) hold where in this case we have
\begin{eqnarray}\label{S5As}
A_{0,0}(\tau,\zeta)&=&\frac1\tau,\qquad A_{0,1}(\tau,\zeta)=\frac1{\tau\left(\tau-\zeta\right)},\nonumber\\
A_{s+1,n}(\tau,\zeta)&=&\frac{\tau^{\beta-1}}{\zeta-\tau}
\frac{\id}{\id\tau}\left(\tau^{1-\beta}A_{s,n}(\tau,\zeta)\right),
\end{eqnarray}
for $s=0,1,2,\dots$, $n=0,1$. Hence, in the case that $|\zeta|<r$ the coefficients can be computed numerically via (\ref{a_s_n}).

For more details for this case see \cite[Chapter 22]{Temme15}.

\section*{Acknowledgments}
This research was supported by a research grant (GRANT11863412/70NANB15H221) from the National Institute of Standards and Technology.

\end{document}
